\documentclass[11pt]{article}
\usepackage{latexsym}
\usepackage{amsmath}
\usepackage{amssymb,amsfonts,amscd}
\usepackage{graphicx}
\usepackage{multirow}
\usepackage{mathabx} 
\usepackage{hyperref}
\usepackage{color}
\usepackage{subcaption}
\usepackage{cite}
\usepackage{algorithm}
\usepackage{algorithmic}

\newtheorem{theorem}{Theorem}[section]
\newtheorem{proposition}[theorem]{Proposition}
\newtheorem{corollary}[theorem]{Corollary}

\newtheorem{rem}[theorem]{Remark}

\newtheorem{example}[theorem]{Example}
\newtheorem{mydef}[theorem]{Definition}

\newcommand\qed{{\hspace*{\fill}$\Box$\vskip12pt plus 1pt}}

\newenvironment{proof}{{\noindent\bf Proof.\ }}{\qed}

\newcommand\bC{{\mathbb C}}

\newcommand\bP{{\mathbb P}}

\newcommand\bR{{\mathbb R}}

\DeclareMathOperator*{\trace}{\rm trace}
\DeclareMathOperator*{\rank}{\rm rank}
\DeclareMathOperator*{\conj}{\rm conj}
\DeclareMathOperator*{\imag}{\rm imag}

\begin{document}
\title{Adaptive strategies for solving parameterized systems \\
using homotopy continuation}

\author{
Jonathan D. Hauenstein\thanks{Department of Applied and Computational Mathematics and Statistics,
University of Notre Dame, Notre Dame, IN 46556 (hauenstein@nd.edu, \url{www.nd.edu/\~jhauenst}).
This author was supported in part by 
Army YIP W911NF-15-1-0219, 
Sloan Research Fellowship BR2014-110 TR14, 
NSF grant ACI-1460032, and
ONR N00014-16-1-2722.
}
\and
Margaret H. Regan\thanks{Department of Applied and Computational Mathematics and Statistics,
University of Notre Dame, Notre Dame, IN 46556 (mregan9@nd.edu,
\url{www.nd.edu/\~mregan9}).
This author was supported in part by 
Schmitt Leadership Fellowship in Science and Engineering
and NSF grant ACI-1440607.}
}

\date{\today}

\maketitle

\begin{abstract}
\noindent 
Three aspects of applying homotopy continuation,
which is commonly used to solve parameterized systems of polynomial equations, are investigated.
First, for parameterized systems which are homogeneous,
we investigate options for performing computations
on an adaptively chosen affine coordinate patch.  
Second, for parameterized systems which are overdetermined,
we investigate options for adaptively 
selecting a well-constrained subsystem to restore numerical 
stability.  Finally, since one is typically interested in only computing real solutions for parameterized problems which arise from applications, we investigate a scheme for 
heuristically identifying solution paths which appear to be 
ending at nonreal solutions and truncating them.  
We demonstrate these three aspects on 
two problems arising in computer vision.  


\end{abstract}


\section{Introduction}\label{Sec:Intro}

Parameterized systems of polynomial equations 
arise in many applications including computer vision \cite{Irschara,Leibe,Snavely},
chemistry \cite{ChemApp1,AlKhateeb}, and kinematics \cite{3RSynthesis,Wampler}.  For a general setup, we
assume that~$F(x;p)$ is a system which is
polynomial in the variables \mbox{$x\in\bC^N$} and 
analytic in the parameters \mbox{$p\in\bC^P$}.
Typically, one is interested in 
efficiently computing the solutions for many 
instances of the parameters. For example, in computer vision,
algorithms are used to solve the same system at many parameter values in order to employ the RANSAC algorithm \cite{Kukelova,RANSAC}.

An approach to repeatedly solve many instances
of a parameterized system is to utilize a 
so-called Gr\"obner trace approach, e.g., see \cite{Kukelova}.  That is, one first performs algebraic
manipulation of the equations at a randomly selected parameter value to discover how to reduce the corresponding
system to a Gr\"obner basis from which the solutions can be efficiently extracted.  The identification
of the algebraic manipulation steps form the {\em ab initio phase} 
and is completed ``offline.''
The ``online'' phase is to repeat these same 
manipulations for each given parameter instance.  
This approach can lead to efficient
solvers in computer vision with 
many potential problems including the 
need for specialized software,
a propagation of errors, instability, 
and expensive computations \cite{Kukelova}.

Rather than employ algebraic manipulation, we consider
using homotopy continuation in the form 
of a parameter homotopy \cite{CoeffParam} (see also \cite[Chap. 6]{BHSW:BertiniBook}).
To employ a parameter homotopy, the ``offline'' {\em ab initio phase}
computes all solutions at a randomly selected parameter value, say $p^*$.  The ``online'' {\em parameter homotopy phase}
tracks the solution paths using homotopy continuation 
as~$p^*$ is deformed to the given parameter instance,
say $\hat{p}$, 
which is shown schematically in Figure~\ref{Fig:ParamHom}.

\begin{figure}
	\centering
	\includegraphics[scale = 0.5]{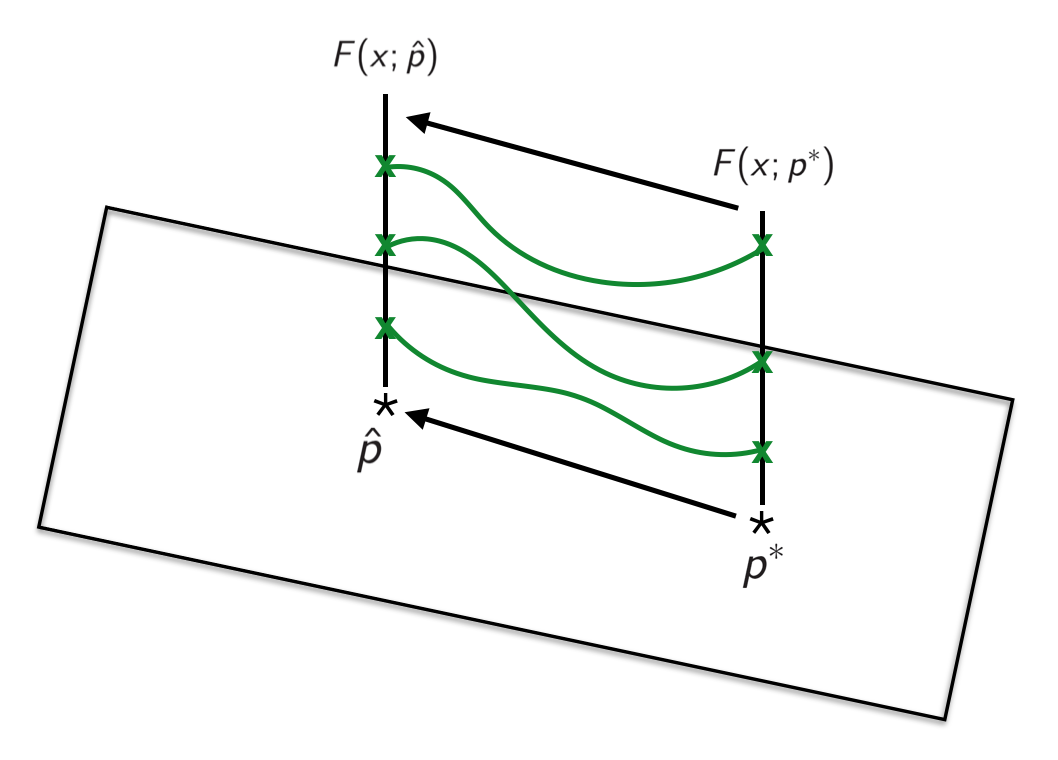}
	\caption{Parameter homotopy phase
starting with the solutions of $F(x;p^*) = 0$
and ending with the solutions of $F(x;\hat{p}) = 0$.}\label{Fig:ParamHom}
\end{figure}

Although parameter homotopies have been used
to solve many instances of a parameterized system \cite{Paramotopy,BHSW:Bertini}, there are three aspects of this 
computation that warrant further consideration.
First, it is common for the parameterized system 
$F(x;p)$ to be homogeneous with respect to the
variables~\mbox{$x$} so that one treats the solutions
as points in the projective space $\bP^{N-1}=\bP(\bC^N)$.
For example, $3\times3$ essential matrices in 3D 
image reconstruction are naturally 
considered as points in \mbox{$\bP^8 = \bP(\bC^{3\times3}) = \bP(\bC^9)$}~\cite{Demazure}.  
For problems which are naturally formulated affinely, 
projective space is used to improve the solving process,
particularly for solutions with large norm
as well as handling nongeneric 
parameter values which have solutions
at infinity \cite{Morgan}.
Computationally, one natural approach for handling 
projective space is to utilize an affine coordinate patch.
We will consider three strategies for selecting 
an affine patch: a fixed coordinate patch which is used throughout the computation \cite{Morgan} (Section~\ref{Sec:GlobalRandomPatch}), 
an adaptive orthogonal patch \cite{ShubSmale} (Section~\ref{Sec:Orthogonal}), and an adaptive coordinate-wise 
patch proposed in Section~\ref{Sec:LocalCoordPatch}.

Second, it is common for 
parameterized problems to be overdetermined.
For example, the set of essential matrices
in computer vision consists of $3\times 3$ matrices 
of rank $2$ where the two nonzero singular values are 
equal.  Since scaling is irrelevant, as mentioned above, 
this set is naturally defined on $\bP^8$ by the 
vanishing of the determinant and the
$9$ cubic Demazure polynomials \cite{Demazure}, namely
\begin{equation}\label{eq:Demazure}
2EE^TE - \trace(EE^T)E = 0.
\end{equation}
This system of 10 polynomials is overdetermined since it
defines an irreducible set of codimension $3$.
Due to the numerical instability of solving 
overdetermined systems, we explore three techniques for reducing down to solving well-constrained systems: a fixed global randomization \cite{NAG} (Section~\ref{Sec:FixedRandomization}), an adaptive pseudoinverse randomization proposed in Section~\ref{Sec:PseudoinverseRandomization}, 
and an adaptive leverage score randomization
proposed in Section~\ref{Sec:LeverageRandomization}.

Third, when solving parameterized problems arising 
from applications, typically only the real solutions
are of interest.  That is, one need not compute
the nonreal endpoints of solution paths defined by a homotopy.
We propose a heuristic strategy in Section~\ref{Sec:Truncation} 
for identifying and truncating paths which appear to be 
ending at nonreal solutions thereby saving computational time.

The remainder of the paper is as follows.  
Section~\ref{Sec:ParameterHomotopies} provides a short introduction
to parameter homotopies and path tracking with more details
provided in \cite{BHSW:BertiniBook,SW:Book}.
Section~\ref{Sec:Patch}
compares the three strategies for affine patches
while Section~\ref{Sec:Randomization} compares
the three strategies for randomizing down
to a well-constrained subsystem.  
Section~\ref{Sec:Algorithm} presents
pseudocode for the path tracking methods with 
Section~\ref{Sec:Truncation} presenting
our heuristic truncation scheme for nonreal solutions.
We compare all of the approaches on two applications
in computer vision in Section~\ref{Sec:Vision}.
The paper concludes in Section~\ref{Sec:Conclusion}.

\section{Parameter homotopies and path tracking}\label{Sec:ParameterHomotopies}

Throughout, we assume that the parameterized system 
$F(x;p)$ is polynomial in the variables $x\in\bC^N$
and analytic in the parameters $p\in\bC^P$.  
This setup ensures that the number 
of nonsingular isolated solutions
of $F = 0$ has a generic behavior with respect to
the parameter space $\bC^P$, e.g., \cite[Thm.~7.1.5]{SW:Book}.
This enables path tracking on the parameter space
via a parameter homotopy \cite{CoeffParam} described below.
We note that one could also consider positive-dimensional components using linear slicing and singular isolated solutions
using deflation techniques, e.g., \cite{Deflation1,IsosingularDeflation},
to reduce to the nonsingular isolated case.  

For generic $p^*\in\bC^P$, the {\em ab initio phase} 
of parameter homotopy continuation is to compute
the isolated nonsingular solutions of $F(x;p^*) = 0$.
This can be accomplished, for example, 
using standard homotopy continuation \cite{BHSW:BertiniBook,SW:Book} which 
is a computation that is performed~once~``offline.''

\begin{figure}[!b]
	\centering
	\includegraphics[scale=0.5]{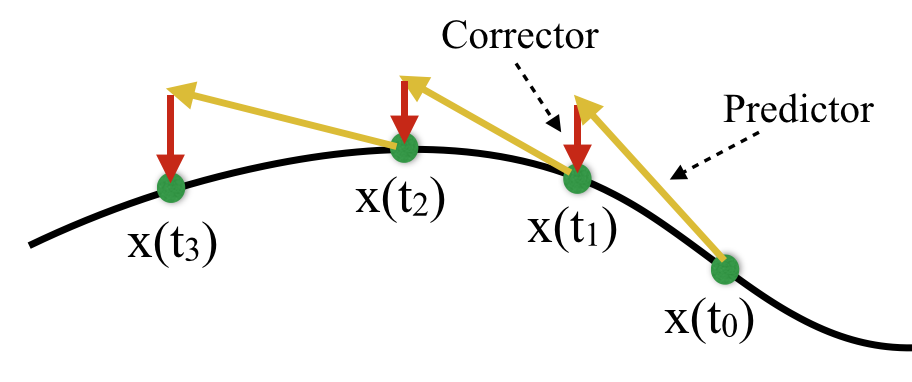}
	\caption{A schematic view of predictor-corrector tracking along a solution path.}\label{fig:Tracking}
\end{figure}

Given $\hat{p}\in\bC^P$, the ``online'' part is 
called the {\em parameter homotopy phase} 
which computes the isolated nonsingular solutions
of $F(x;\hat{p}) = 0$ using the parameter homotopy 
\begin{equation}\label{eq:ParameterHomotopy}
H(x,t) = F(x;tp^* + (1-t)\hat{p}).
\end{equation}
In particular, for each nonsingular solution $x^*$ 
of $F(x;p^*) = 0$, one considers the solution path
$x(t)$ defined by $x(1) = x^*$ and $H(x(t),t) \equiv 0$.
Genericity of $p^*$ ensures that each solution path $x(t)$
is smooth for $t\in(0,1]$ and satisfies
the Davidenko differential equation
\begin{equation}\label{eq:Davidenko}
J_x H(x,t)\cdot\dot{x}(t) = -J_t H(x,t)
\end{equation}
where $J_x H(x,t)$ and $J_t H(x,t)$ are the Jacobian
matrix with respect to $x$ and Jacobian vector
with respect to $t$, respectively.
Hence, one can employ a predictor-corrector tracking strategy
starting with the initial value $x(1) = x^*$
to compute $x(0)$.  
A schematic view of predictor-corrector
tracking is provided in Figure~\ref{fig:Tracking}
with more details provided in \cite{BHSW:BertiniBook,SW:Book}.
In particular, the predictor follows from the 
differential equation~\eqref{eq:Davidenko} 
while the corrector uses the fact that 
$H(x(t),t) \equiv 0$.  Since, for $t\in(0,1]$,
$x(t)$ is a nonsingular isolated solution of
$H(\bullet,t) = 0$, Newton's method is 
locally quadratically convergent.
Hence, our computations will utilize classical
$4^{\rm th}$ order Runge-Kutta prediction method 
with the corrector being several iterations Newton's method.

\section{Affine patches}\label{Sec:Patch}

The projective space $\bP^N$ is the set of lines
in $\bC^{N+1}$ passing through the origin.  
In particular, there is a choice to be made 
for performing computations on $\bP^N$ due
to selecting a representation of each point in $\bP^N$.
One standard approach is to utilize
an affine coordinate patch where a Zariski open dense 
subset of $\bP^N$ is represented by a hyperplane in $\bC^{N+1}$
as illustrated in the following.

\begin{example}\label{ex:TwistedCubic}
Consider intersecting the twisted cubic curve $C\subset\bP^3$ 
with the hyperplane defined by $x_0 + x_1 + x_2 + x_3 = 0$,
namely computing the three solutions on $\bP^3$ 
of the polynomial system
\begin{equation}\label{eq:TwistedCubic}
f(x) = \left[\begin{array}{c} 
x_0 x_2 - x_1^2 \\ x_1 x_2 - x_0 x_3 \\ x_1 x_3 - x_2^2 \\
x_0 + x_1 + x_2 + x_3 \end{array}\right] = 0
\end{equation}
which are 
\begin{equation}\label{eq:ThreePoints}
[1, -1,  1, -1],~[ 1,  i, -1, -i],~[ 1, -i, -1,  i] \in\bP^3
\end{equation}
where $i = \sqrt{-1}$.  
In \eqref{eq:ThreePoints}, each point in $\bP^3$ 
is represented by a unique vector in $\bC^4$ 
using the affine coordinate patch
defined by \mbox{$x_0 = 1$}, i.e., 
represented uniquely in the form $[1,x_1,x_2,x_3]$.
The set of points in $[x_0,x_1,x_2,x_3]\in\bP^3$ which 
cannot be represented in this way is the hyperplane $x_0 = 0$.
\end{example}

The key to selecting an affine coordinate patch
is to make sure that every projective point of interest,
e.g., every point along every homotopy solution path,
has a representation in that affine patch.
For example, the first point in \eqref{eq:ThreePoints} 
cannot be represented in the affine coordinate
patch defined by $x_0 + x_1 = 1$.  

The following describes three strategies for selecting an
affine patch.  The first uses a fixed general affine patch \cite{Morgan} while the second and third utilize a locally adapted orthogonal \cite{ShubSmale} and coordinate-wise patching strategy, respectively.

\subsection{Fixed general affine patch}\label{Sec:GlobalRandomPatch}

The approach presented in \cite{Morgan} uses a general 
affine coordinate patch.  
That is, for a general $v\in\bC^{N+1}$,
one performs all computations on
the fixed affine coordinate patch defined by 
$$v\cdot x = v^Hx = 1$$
where $v^H$ is the Hermitian (conjugate) transpose of $v$.
See Figure~\ref{fig:Patching}(a) for a schematic view.

\begin{figure}[!t]
\centering
\begin{subfigure}{.3\textwidth}
  \centering
  \includegraphics[scale = 0.5]{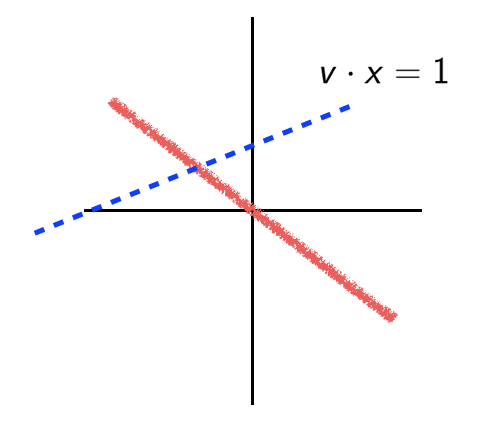}
\end{subfigure}
\begin{subfigure}{.3\textwidth}
  \centering
  \includegraphics[scale = 0.5]{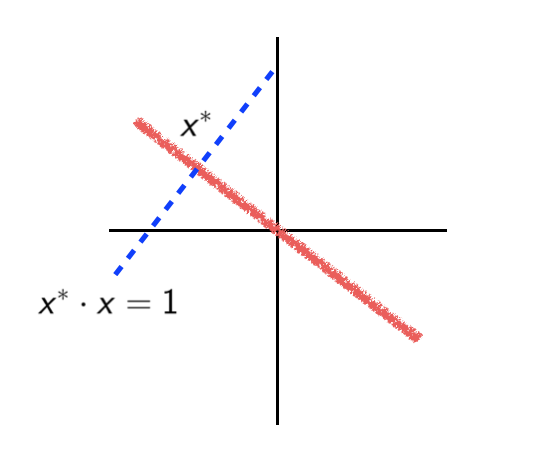}
\end{subfigure}
\begin{subfigure}{.3\textwidth}
  \centering
  \includegraphics[scale = 0.5]{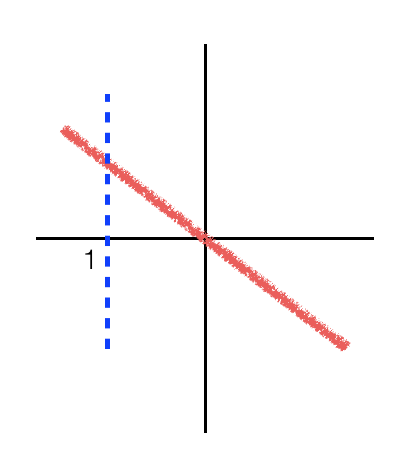}
\end{subfigure} \\
(a) \hspace{1.7in} (b)\hspace{1.7in}  (c)
\caption{Schematic drawing of (a) fixed affine patch, (b) orthogonal affine patch, and (c)~coordinate-wise affine patch.}
\label{fig:Patching}
\end{figure}

The advantage of using a fixed affine coordinate patch
is that it is chosen at the beginning and is
fixed throughout the computations.  
Thus, one could perform computations implicitly
on the patch which removes one of the variables, 
e.g., remove the variable
$x_0$ via
$$x_0 = \frac{1}{\conj(v_0)}\left(1 - \conj(v_1) x_1 - \cdots - \conj(v_N) x_N\right).$$

The disadvantage of using a fixed patch is that
ill-conditioning could artificially be introduced
as shown in the following.

\begin{example}\label{ex:TwistedCubic2}
With $f(x)$ as in \eqref{eq:TwistedCubic}, consider
the system of equations
\begin{equation}\label{eq:TwistedCubicG}
g(x) = \left[\begin{array}{c} f(x) \\
v \cdot x - 1 \end{array}\right] = 0.
\end{equation}
Table~\ref{tab:CN} shows the condition number
with respect to the $2$-norm 
of the Jacobian matrix of~$g$ using various 
vectors $v$ at the solution corresponding to
$[1,-1,1,-1]$ in the respective
affine patches.

\begin{table}[!hb]
\centering
\begin{tabular}{c|c|c|c}
$v$ & $(1,0,0,0)$ & $(0.8695,0.4670,-0.0231,0.1592)$ & 
$(0.1947,0.3999,-0.5268,-0.7243)$ \\
\hline
CN & 10.2 & 158.2 & 113,574.2 
\end{tabular}
\caption{Condition number of the Jacobian matrix of $g$
with respect to different coordinate patches}
\label{tab:CN}
\end{table}
\end{example}

One can attempt to limit this artificial ill-conditioning
by using a locally selected 
affine coordinate patch, i.e.,
one which is adapted to the current point to
be represented.  The next two subsections 
consider two methods for selecting local patches.

\subsection{Orthogonal affine patches}\label{Sec:Orthogonal}

In \cite{ShubSmale}, computations are performed
locally in the Hermitian orthogonal complement
of a point in projective space
which Shub and Smale say can be 
considered as the tangent space of $\bP^N$.
To fix notation, assume that $x^*\in\bC^{N+1}\setminus\{0\}$
such that $[x^*]\in\bP^N$ is the current point
in projective space under consideration.  
To avoid computations on vectors which are too
large or too small, we will assume that 
$x^*\cdot x^* = \|x^*\|_2^2 = 1$.  Thus,
with this setup, the orthogonal affine patch is 
defined by
$$x^*\cdot x = 1.$$
If $y^*\in\bC^{N+1}$ is the another point on this affine 
patch, let $\Delta x = y^* - x^*$.  Thus, 
$\Delta x \cdot x^* = 0$ so 
that~$\Delta x$ is orthogonal to $x^*$ giving the method its name.
See Figure~\ref{fig:Patching}(b) for a schematic view.

\begin{example}\label{ex:TwistedCubic3}
Let $g(x)$ be as in \eqref{eq:TwistedCubicG}
and $x^* = (1/2,-1/2,1/2,-1/2)$
so that $x^*\cdot x^* = 1$ and $[x^*]=[1,-1,1,-1]$.
Then, the condition number of the Jacobian matrix
of $g$ with respect to the~$2$-norm 
with the affine patch defined by $x^*\cdot x = 1$
is $4.37$.
\end{example}

When path tracking, one uses an orthogonal affine
coordinate patch based on the current point on the path 
and performs a predictor-corrector step in that patch.
If the step is successful, the patch is updated
based on the new point on the path (see Section~\ref{Sec:Algorithm}).

The advantage of using an orthogonal affine coordinate
patch is the typically well-controlled condition number.  
When using a fixed patch as in Section~\ref{Sec:GlobalRandomPatch},
one can globally remove a variable.  With a locally
adapting patch, one is able to locally remove a variable
which will typically depend upon all of the other variables.
The next method uses a locally adapted patch that fixes
one~variable.

\subsection{Coordinate-wise affine patches}\label{Sec:LocalCoordPatch}

Coordinate-wise affine patches have the form $x_j = 1$
for some $j\in\{0,\dots,N\}$.  For example, 
the points in \eqref{eq:ThreePoints}
from Ex.~\ref{ex:TwistedCubic} are represented
using the coordinate-wise patch $x_0 = 1$.
The advantage of using such a coordinate-wise 
patch is the simplicity of setting a coordinate equal
to $1$.  One disadvantage could be having
the other coordinates be large
if the point is ``near'' the hyperplane 
at ``infinity,'' i.e., $x_j = 0$.
To overcome this, we locally adapt the selection 
of the coordinate $j$.  
That is, if $x^*\in\bC^{N+1}\setminus\{0\}$ corresponds
to $[x^*]\in\bP^N$, we can assume that $\|x^*\|_\infty = 1$
and select one coordinate $j$ such that $x^*_{j} = 1$.
Hence, the corresponding coordinate-wise
affine patch is 
$e_j\cdot x = 1$ where $e_j$
is the $j^{\rm th}$ standard coordinate vector.  
See Figure~\ref{fig:Patching}(c) for a schematic view.

\begin{example}\label{ex:TwistedCubic4}
Let $g(x)$ be as in \eqref{eq:TwistedCubicG}
with $[x^*] = [1,-1,1,-1]$.  
The condition number of the Jacobian matrix
of $g$ with respect to the $2$-norm 
is $10.2$ when 
using either $x_0 = 1$ or $x_3 = 1$ 
and $8.5$ when using either $x_1 = 1$ or $x_2 = 1$.
\end{example}

As with the orthogonal patch in Section~\ref{Sec:Orthogonal},
when path tracking, we utilize a local strategy which updates 
the coordinate $j$ defining the 
affine patch after each successful step (see Section~\ref{Sec:Algorithm}).
Although the condition number is typically not as small
as the orthogonal case, it is trivial 
to remove a variable since $x_j = 1$ which helps to reduce
the cost of linear algebra in taking a step.

\subsection{Optimal patching}\label{Sec:OptimalPatch}

The affine patches described in 
Sections~\ref{Sec:GlobalRandomPatch}-\ref{Sec:LocalCoordPatch}
are of the form $v\cdot x = 1$
for a vector $v\in\bC^{N+1}$.  Although one would like to minimize
the condition number of the Jacobian over all such vectors
\mbox{$v\in\bC^{N+1}$} for a given solution, we will efficiently approximate solving this large optimization problem 
by considering rescalings.  The following
demonstrates that this can yield improvements.

\begin{example}\label{ex:TwistedCubicPatch}
Reconsider the setup from Ex.~\ref{ex:TwistedCubic4}.
Fix $z^* = (1,-1,1,-1)$ and $\alpha = e_0$.
For \mbox{$\lambda \in \bC\setminus\{0\}$},
consider $v = \lambda \cdot \alpha$ and $x^* = z^* / \lambda$ 
so that the patch is simply defined by
$v\cdot x = \lambda\cdot x_0 = 1$, i.e., $x_0 = 1/\lambda$.
When $\lambda = 1$, Ex.~\ref{ex:TwistedCubic4}
showed that the condition number is $10.2$
with Figure~\ref{fig:Rescale} showing 
the condition number
as a function of $\lambda$.  When $\lambda = 1.7$,
the condition number decreases to $7.2$.  

\begin{figure}[!t]
\centering
  \includegraphics[scale = 0.28]{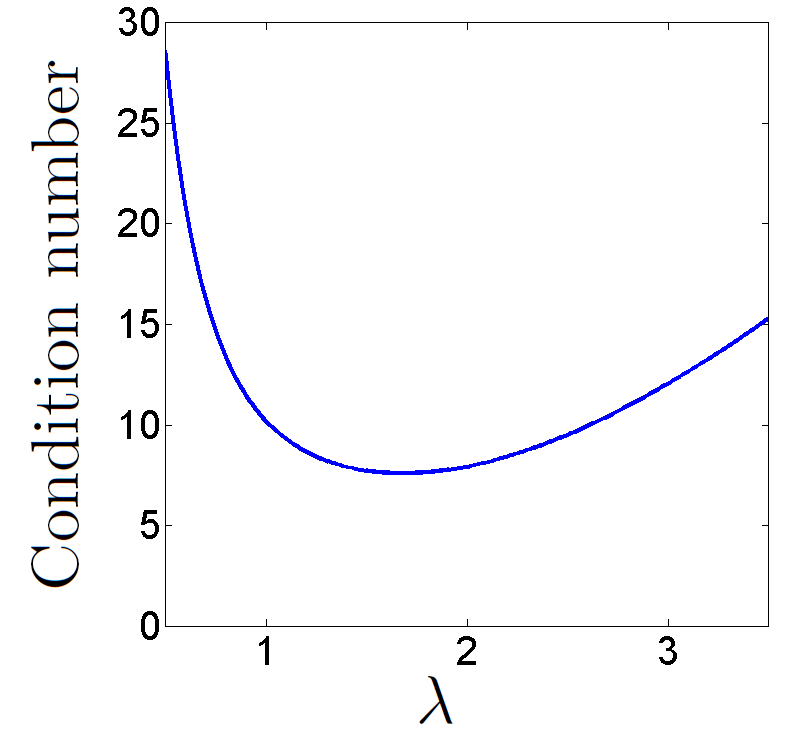}
\caption{Plot of condition number as a function of the scaling parameter $\lambda$.}
\label{fig:Rescale}
\end{figure}
\end{example}

Let $f(x)$ be a system of $N$ polynomials, each of degree $d > 0$, defined on $\bP^N$, 
$\alpha\in\bC^{N+1}\setminus\{0\}$, $\lambda > 0$,
and affine patch $v\cdot x = 1$ where $v = \lambda\cdot\alpha$.
Suppose that $z^*\in\bC^{N+1}$ with $\alpha\cdot z^* = 1$ 
such that $[z^*]\in\bP^N$ solves $f = 0$.
We aim to select the scaling $\lambda$ to
improve the conditioning at $x^* = z^*/\lambda$.
Since each $f_i$ is homogeneous of degree $d > 0$,
each entry of the gradient $\nabla f_i$ is 
either~$0$ or homogeneous of degree $d-1$.  Hence,
consider the $(N+1)\times(N+1)$ matrix
\begin{equation}\label{eq:M}
M(\lambda) = \lambda^{d-1} \left[\begin{array}{c}
\nabla f_1(z^*/\lambda) \\ \vdots \\
\nabla f_N(z^*/\lambda) \\ 
\lambda\cdot \alpha^H 
\end{array}\right]
= \left[\begin{array}{c}
\nabla f_1(z^*) \\ \vdots \\
\nabla f_N(z^*) \\ 
\lambda^d\cdot \alpha^H 
\end{array}\right]
= \left[\begin{array}{c} Jf(z^*) \\ \lambda^d\cdot \alpha^H
\end{array}\right]
\end{equation}
where $Jf(z^*)$ is the Jacobian matrix of $f$ evaluated at $z^*$.
To further simplify the computation, we consider
minimizing $\kappa_{\infty,1}(M(\lambda)) = \|M(\lambda)\|_\infty\cdot\|M(\lambda)^{-1}\|_1$ 
which the following shows can be minimized 
using data from $M(1)$ and~$M(1)^{-1}$.

\begin{theorem}\label{thm:Minimize}
If $M(\lambda)$ from \eqref{eq:M} is written
as $M(\lambda) = \left[\begin{array}{c} J \\ \lambda^d\cdot\alpha^H \end{array}\right]$
and $M(1)^{-1} = \left[\begin{array}{cc} K & \beta \end{array}\right]$ where $J\in\bC^{N\times(N+1)}$, $K\in\bC^{(N+1)\times N}$, and $\alpha,\beta\in\bC^{N}$, then $\kappa_{\infty,1}(M(\lambda))$
is minimized when 
$$\lambda = \sqrt[2d]{\frac{\|J\|_\infty\cdot\|\beta\|_1}{\|K\|_1\cdot\|\alpha\|_1}}.$$
\end{theorem}
\begin{proof}
From $M(1)^{-1}$, it is easy to verify
that $M(\lambda)^{-1} = \left[\begin{array}{cc} K & \beta/\lambda^d \end{array}\right]$.  Hence,
$$\|M(\lambda)\|_\infty = \max\{\|J\|_\infty,~\lambda^d\cdot\|\alpha\|_1\}\hbox{~~~~and~~~~}
\|M(\lambda)^{-1}\|_1 = \max\{\|K\|_1,~\|\beta\|_1/\lambda^d\}
$$
so that
$$\kappa_{\infty,1}(M(\lambda)) = \max\{\|J\|_\infty\cdot\|\beta\|_1/\lambda^d,~\|J\|_\infty\cdot \|K\|_1,~\|\alpha\|_1\cdot\|\beta\|_1,~\lambda^d\cdot \|K\|_1\cdot\|\alpha\|_1\}.$$
Hence, $\kappa_{\infty,1}(M(\lambda))$ is a convex function 
such that $\kappa_{\infty,1}(M(\lambda)) = \|J\|_\infty\cdot \|\beta\|_1/\lambda^d$ for $0 < \lambda \ll 1$
and $\kappa_{\infty,1}(M(\lambda)) = \lambda^d\cdot\|\alpha\|_1\cdot\|K\|_1$ for $\lambda \gg 1$.  
Figure~\ref{fig:QuadSystem} presents an
example plot of $\kappa_{\infty,1}(M(\lambda))$.
In particular, we have that
the minimum is achieved when 
$$\|J\|_\infty\cdot \|\beta\|_1/\lambda^d = \lambda^d\cdot\|\alpha\|_1\cdot\|K\|_1$$
which occurs when 
$\lambda = \sqrt[2d]{\frac{\|J\|_\infty\cdot\|\beta\|_1}{\|K\|_1\cdot\|\alpha\|_1}}$.  In particular, the minimum 
of $\kappa_{\infty,1}$ is 
$$\max\{\sqrt{\|J\|_\infty\cdot\|K\|_1\cdot\|\alpha\|_1\cdot\|\beta\|_1},~\|J\|_\infty\cdot \|K\|_1,~\|\alpha\|_1\cdot\|\beta\|_1\}.$$
\end{proof}

\begin{example}\label{ex:QuadSystem}
To illustrate Theorem~\ref{thm:Minimize}, consider the polynomial system
$$f(x) = \left[\begin{array}{c}
x_0 x_2 - x_1^2 \\
x_0^2 + x_1^2 + x_2^2 - x_3^2 \\
x_1 x_2 + x_1 x_3 - x_0^2
\end{array}\right]$$
using the coordinate-wise patch $x_3 = 1$ which is
defined by
$\alpha = (0,0,0,1)$ with solution
$$z^* = ((\sqrt{5}+1)/4,~1/2,~(\sqrt{5}-1)/4,~1)\approx
(0.8090,~0.5,~0.3090,~1).$$  
Following the notation of Thm.~\ref{thm:Minimize},
we have
$$d = 2,~~\|J\|_\infty = 5.2361,~~\|K\|_1 = 1.2361,~~\|\alpha\|_1 = 1,\hbox{~~and~~}\|\beta\|_1 = 2.6180$$
so that
$$\lambda = \sqrt[4]{\frac{\|J\|_\infty\cdot\|\beta\|_1}{\|K\|_1\cdot\|\alpha\|_1}} \approx 1.8249$$
is the scaling factor to minimize $\kappa_{\infty,1}$ as
shown in Figure~\ref{fig:QuadSystem}.  Hence, we take the affine
patch defined by $v\cdot x = 1$ where $v = \lambda\cdot\alpha \approx 
(0,0,0,1.8249)$ yielding the corresponding point 
$$x^* = z^*/\lambda \approx (0.4433,~0.2740,~0.1693,~0.5480).$$

\begin{figure}[!t]
\centering
  \includegraphics[scale = 0.45]{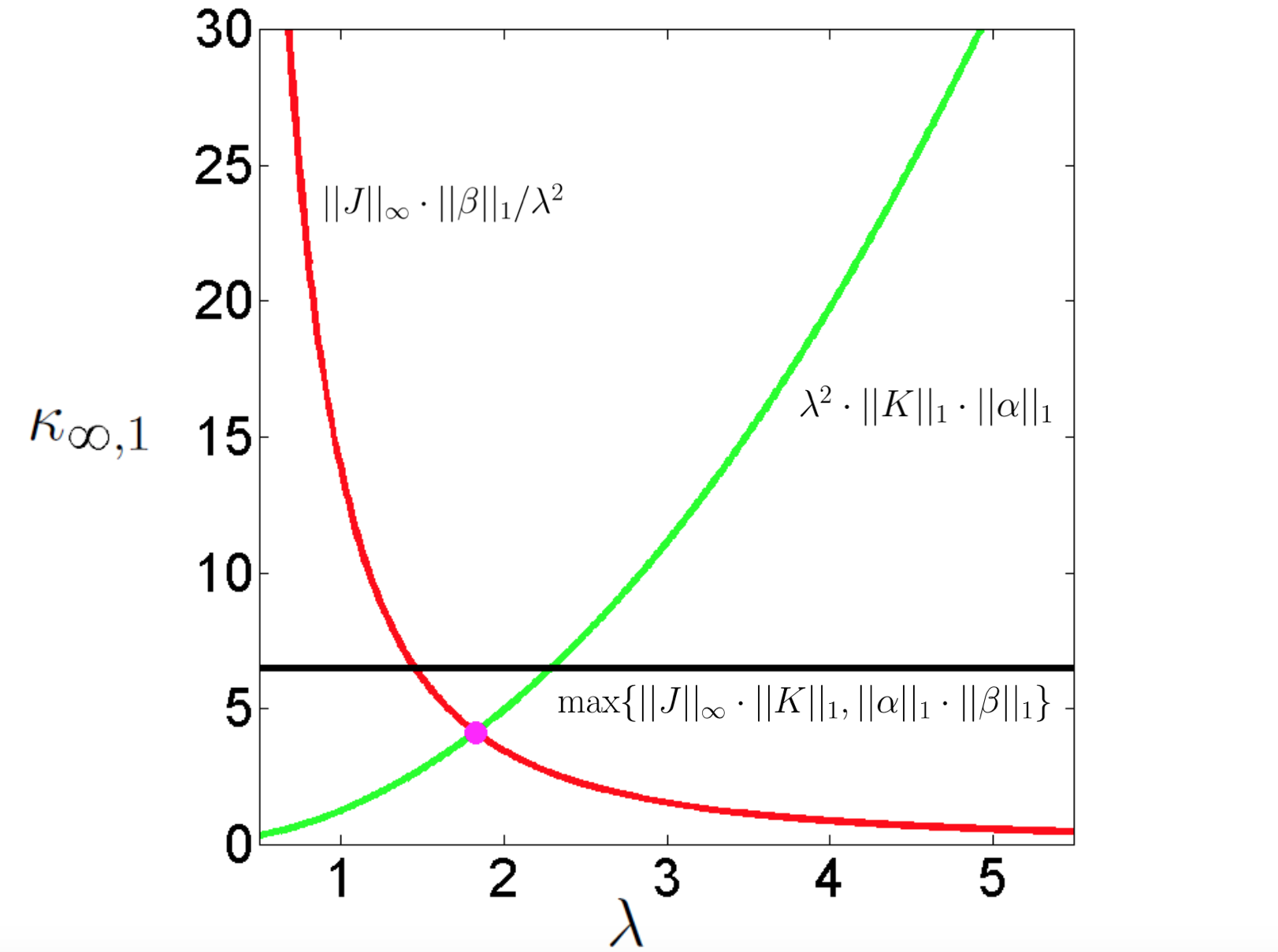}
\caption{Plot of $\kappa_{\infty,1}$ a function of the scaling parameter $\lambda$.}
\label{fig:QuadSystem}
\end{figure}
\end{example}

When using an orthogonal patch, the following shows that we can simplify the computation~of~$\lambda$.

\begin{corollary}\label{cor:Minimize}
Using the same setup as Theorem~\ref{thm:Minimize},
if $\alpha = z^*$ with $\alpha^H\alpha = \|\alpha\|_2^2 = 1$, then
$\beta = \alpha = z^*$ so that~$\kappa_{\infty,1}(M(\lambda))$
is minimized when 
$$\lambda = \sqrt[2d]{\frac{\|J\|_\infty}{\|K\|_1}}.$$
\end{corollary}
\begin{proof}
Following the notation from 
\eqref{eq:M} and Thm.~\ref{thm:Minimize},
$M(1) = \left[\begin{array}{c} J \\ \alpha^H \end{array}\right]$ and $M(1)^{-1} = \left[\begin{array}{cc} K & \beta \end{array}\right]$.  Euler's Theorem yields that 
$J \alpha = 0$ since each $f_i$ is homogeneous and $f(z^*) = 0$.
Since $\beta$ is the unique vector such that
$J\beta = 0$ and $\alpha^H \beta = 1$,
we have $\alpha = \beta$ and the result follows from Thm.~\ref{thm:Minimize}.
\end{proof}

\begin{example}\label{ex:QuadSystem2}
Reconsider $f$ from Ex.~\ref{ex:QuadSystem} with the
orthogonal patch $\alpha \cdot x = 1$ where $\alpha = z^*$ and
$$z^* = ((\sqrt{5}+1)/\sqrt{32},~1/\sqrt{8},~(\sqrt{5}-1)/\sqrt{32},~1/\sqrt{2})\approx
(0.5721,~0.3536,~0.2185,~0.7071).$$
By Cor.~\ref{cor:Minimize}, $\alpha = \beta$ and we have
$$d = 2,~~\|J\|_\infty = 3.7025,\hbox{~~and~~}\|K\|_1 = 1.7481$$
so that
$$\lambda = \sqrt[4]{\frac{\|J\|_\infty}{\|K\|_1}} \approx 1.2064$$
is the scaling factor to minimize $\kappa_{\infty,1}$.
Hence, we take the affine patch defined by 
$v\cdot x = 1$ where $v = \lambda\cdot\alpha \approx (0.6901,~0.4265,~0.2636,~0.8530)$ yielding the corresponding point
$$x^* = z^*/\lambda \approx (0.4742,~0.2931,~0.1811,~0.5861).$$
Figure~\ref{fig:QuadSystem2} compares $\kappa_{\infty,1}(M(\lambda))$ with the condition number 
$\kappa_2(M(\lambda))$.

\begin{figure}[!t]
\centering
  \includegraphics[scale = 0.45]{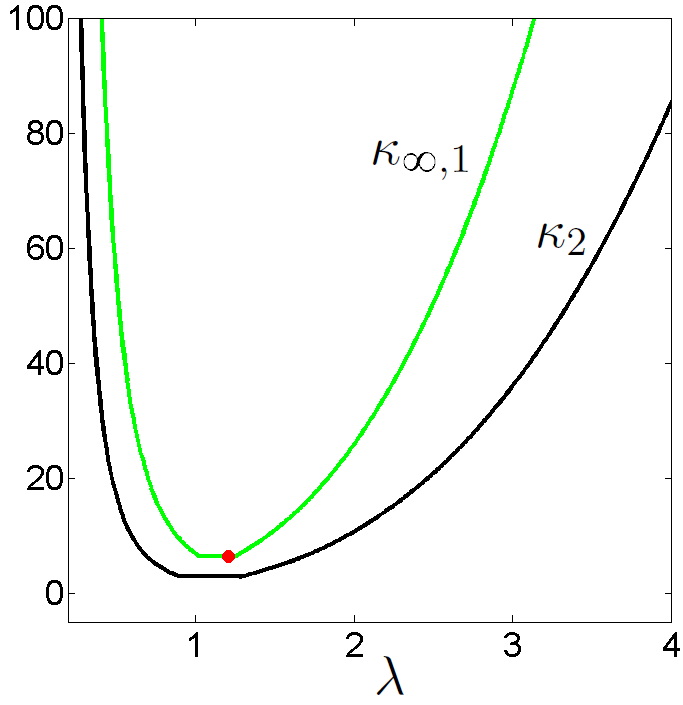}
\caption{Plot of $\kappa_{\infty,1}$ and $\kappa_2$
as a function of the scaling parameter $\lambda$.}
\label{fig:QuadSystem2}
\end{figure}
\end{example}

\section{Randomization}\label{Sec:Randomization}

The polynomial system $f(x)$ in Ex.~\ref{ex:TwistedCubic}
is overdetermined since it consists of $4$ polynomials
defined on~$\bP^3$.  If one considers appending
the patch equation $x_0 - 1 = 0$ to $f(x)$, the system
is still overdetermined with $5$ polynomials defined 
on $\bC^4$.  The first $3$ polynomials
define the twisted cubic curve, which has codimension $2$,
while the fourth polynomial is a hyperplane
that intersects the twisted cubic curve transversely.
Since a general perturbation applied to the first 
$3$ polynomials results in a system with no solutions,
numerically solving inexact 
overdetermined systems is unstable, 
e.g., see~\cite[\S~9.2]{BHSW:BertiniBook}.  
One could recover stability by using 
Gauss-Newton least-squares approaches, e.g., see \cite{Deuflhard}. 
Another approach for stabilization 
is to replace overdetermined systems
with well-constrained subsystems.  The following
is a version of 
Bertini's Theorem, e.g., see \cite[Thm.~A.8.7]{SW:Book}
and \cite[Thm.~9.3]{BHSW:BertiniBook}, which
permits such replacements using well-constrained subsystems.

\begin{theorem}[Bertini's Theorem]\label{thm:Bertini}
If $f(x)$ is a system of $n$ polynomials 
on $\bC^N$ and $1\leq k\leq n$, then there exists 
a Zariski open dense set $U \subset \bC^{k \times n}$
such that for every $A\in U$, each generically nonsingular
irreducible component of the solution set of 
$f = 0$ of codimension at most $k$
is a generically nonsingular
irreducible component of the solution set of $A\cdot f = 0$.
\end{theorem}

We state Bertini's Theorem to focus on generically nonsingular
irreducible components for path tracking purposes (see Section~\ref{Sec:Algorithm}).  One can always reduce to this case
using deflation, e.g.,~\cite{Deflation1,IsosingularDeflation}.
We note that randomization could also add extra solutions
which do not solve the original system.  
Such extraneous solutions 
can be certifiably identified~\cite{alphaCertified}.

\begin{example}\label{eq:TwistedCubicRandomization}
For $f(x)$ as in \eqref{eq:TwistedCubic} together with the patch equation $x_0 - 1 = 0$, consider the system
\begin{equation}\label{eq:TwistedCubicRand}
g(x) = \left[\begin{array}{ccccc} 
2   & -1  &  -3   &  2   &  2 \\
-2  &  -1 &    0  &   3  &  -4 \\
5   &  3  &  -1   & -2   & -4 \\
-5  &   3  &   2  &   2  &   0\end{array}\right]\cdot \left[\begin{array}{c} f(x) \\ x_0 - 1 \end{array}\right] = 0
\end{equation}
which is a well-constrained system consisting 
of $4$ polynomial equations on $\bC^4$.
Bertini's Theorem yields that the three points corresponding
to \eqref{eq:ThreePoints} are isolated nonsingular
solutions to $g = 0$.  
Randomization has also added two additional solutions to $g = 0$,
approximately
$$(0.7955 \pm 0.0744i,~0.3755 \mp 0.6315i,
~-1.2239 \mp 0.1598i,~-0.6730 + 0.9810i)$$
where $i = \sqrt{-1}$ which are easily identified since $x_0 \neq 1$.
\end{example}

In \eqref{eq:TwistedCubicRand}, we selected the randomizing matrix $A$ to have small integer entries for presentation purposes.  
In practice, the matrix $A$ is selected to have random complex entries.

Analogous to the patching strategies described 
in Section~\ref{Sec:Patch},
we describe three randomization strategies.  
The first is based directly on Theorem~\ref{thm:Bertini}
which utilizes a fixed randomization matrix, 
a commonly used technique in numerical algebraic
geometry computations, e.g., \cite[\S~9.2]{BHSW:BertiniBook}.
The second utilizes a locally adapted orthogonalization strategy based on the Moore-Penrose pseudoinverse.  
To create sparse randomizations, the third
utilizes a locally adapted leverage score strategy.

\subsection{Fixed randomization}\label{Sec:FixedRandomization}

As described in Bertini's Theorem (Theorem~\ref{thm:Bertini}),
one can utilize a fixed general randomization matrix
$A\in\bC^{k\times n}$ to yield a well-constrained system.  
As detailed in \cite[\S~13.5]{SW:Book},
one could take $A$ to have the form $A = [I~Q]$
where $I$ is the $k\times k$ identity matrix and
$Q\in\bC^{k\times (n-k)}$ is general to reduce
the number of computations needed to randomize the system.

Similar to the fixed affine coordinate patch in Section~\ref{Sec:GlobalRandomPatch}, the disadvantage
of using a fixed randomization is that ill-conditioning
can artificially be introduced as demonstrated in the following.

\begin{example}\label{ex:TwistedCubic6}
With $f(x)$ as in \eqref{eq:TwistedCubic}, consider
the system $h:\bC^4\rightarrow\bC^5$ where
\begin{equation}\label{eq:TwistedCubicx0}
h(x) = \left[\begin{array}{c} f(x) \\
x_0 - 1 \end{array}\right]
\end{equation}
and isolated nonsingular solution $(1,-1,1,-1)$.  
Table~\ref{tab:CN2} shows the condition number
with respect to the $2$-norm of the Jacobian matrix
for various randomizations of the form $[I~Q]\cdot h$
where $I$ is the $4\times 4$ identity matrix
and $Q\in\bC^{4\times 1}$ at this isolated nonsingular
solution.  

\begin{table}[!ht]
\centering
\begin{tabular}{c|c|c|c}
$Q$ & $[1,1,1,1]^T$ & $[-0.0109,0.5208,0.4013,0.7534]^T$ & 
$[-0.0889,0.6266,0.7152,0.2966]^T$ \\
\hline
CN & 33.3 & 185.6 & 67,193.2 
\end{tabular}
\caption{Condition number of the Jacobian matrix of 
$[I~Q]\cdot h$ for different choices of $Q$}
\label{tab:CN2}
\end{table}
\end{example}

One can attempt to limit this artificial ill-conditioning
by using a randomization which is locally
adapted to the current solution under consideration.
The next two subsections 
consider two methods for selecting local randomizations.

\subsection{Pseudoinverse randomization}\label{Sec:PseudoinverseRandomization}

At a nonsingular solution, 
the Jacobian matrix has full rank
so that one can use 
the Moore-Penrose pseudoinverse to construct a 
randomization matrix.
To that end, assume that $f:\bC^N\rightarrow\bC^n$ is a polynomial system and $x^*\in\bC^N$ is an isolated nonsingular solution of $f = 0$, i.e., $f(x^*) = 0$ and $\rank Jf(x^*) = N \leq n$
where $Jf(x^*)$ is the Jacobian matrix of $f$ evaluated
at $x^*$.  Via the singular value decomposition, we can
find unitary matrices 
$U\in\bC^{n\times N}$ and $V\in\bC^{N\times N}$
and invertible diagonal matrix $\Sigma\in\bR^{N\times N}$ 
such that
$$Jf(x^*) = U\cdot \Sigma \cdot V^H \in \bC^{n\times N}$$
where $V^H$ is the Hermitian (conjugate) transpose 
of $V$.  The Moore-Penrose pseudoinverse of $Jf(x^*)$ is $Jf(x^*)^\dagger = V\cdot \Sigma^{-1}\cdot U^H\in\bC^{N\times n}$ 
so that
$$Jf(x^*)^\dagger\cdot Jf(x^*) = I$$
where $I$ is the $N\times N$ identity matrix.
Therefore, the randomized well-constrained subsystem
$$g(x) = Jf(x^*)^\dagger\cdot f(x)$$
has a nonsingular solution at $x^*$ with $Jg(x^*) = I$.

\begin{example}\label{ex:TwistedCubic7}
With $h:\bC^4\rightarrow\bC^5$ as in \eqref{eq:TwistedCubicx0}
and $x^* = (1,-1,1,-1)$, we have
$$Jh(x^*)^\dagger = \left[\begin{array}{ccccc}
0 & 0 & 0 & 0 & 1 \\
1/6 & 1/3 & 1/6 & 1/2 & -1 \\
1/3 & -1/3 & -2/3 & -1 & 1 \\
-1/2 & 0 & 1/2 & 3/2 & -1 
\end{array}\right].$$
Thus, for the randomized well-constrained subsystem
$g(x) = Jh(x^*)^\dagger\cdot h(x)$,
$g(x^*) = 0$ and 
$Jg(x^*) = I \in \bC^{4\times 4}$.
\end{example}

When path tracking, one uses the pseudoinverse
based on the current point on the path 
and performs a predictor-corrector step with that randomized
system.  If the step is successful, the randomization 
is updated based on the new point on the path (see Section~\ref{Sec:Algorithm}).

The advantage of using the pseudoinverse randomization 
is that the Jacobian matrix at the current point
of the randomized system is the identity matrix, a perfectly
conditioned matrix.  The disadvantage 
is the extra computations:
both in the computation of the pseudoinverse 
and to utilize the randomization which is typically dense.
The following uses a sparse randomization.

\subsection{Leverage score randomization}\label{Sec:LeverageRandomization}

The randomizations in Sections~\ref{Sec:FixedRandomization}
and~\ref{Sec:PseudoinverseRandomization} construct
a new system which typically 
depends upon all of the polynomials in the 
original system.  We aim to design an approach
to select a well-constrained subset 
which has a nonsingular solution at the current point.
For example, if the given system is vastly 
overdetermined, we aim to select a subset of polynomials
from the system rather than having to randomize together all
of the polynomials.  
The method that we propose is based on {\em leverage scores}
which were originally used to find outliers in
data when computing regression analysis \cite{Leverage}.  
We follow the approach in \cite{Ipsen} which states
that leverage scores can be used to describe important data
in a matrix.  In our case, we aim to locate
polynomials in the system 
whose gradients are important
rows of the Jacobian matrix evaluated at the given point.

\begin{mydef}[Leverage scores]\label{Def:LeverageScores}
For a matrix $M\in\bC^{m\times n}$ of rank $n \leq m$,
let $Q\in\bC^{m\times n}$ be any unitary matrix 
whose columns form a basis for the column span of $M$.
Then, the {\em leverage scores} $\ell_1,\dots,\ell_m\in\bR_{\geq0}$ for $M$ are defined by
$$
\ell_j = \|Q_{j}\|_2^2
$$
where $Q_{j}$ is the $j^{\rm th}$ row of $Q$.
\end{mydef}

The definition of leverage scores is basis independent
\cite[\S~5.1]{Ipsen} so that that leverage scores
are well-defined.  Moreover, since $Q$ is unitary,
each $\ell_j\in[0,1]$ with $\sum_{j=1}^m \ell_j = n$.

As above, assume that $f:\bC^N\rightarrow\bC^n$ is a polynomial system and $x^*\in\bC^N$ is an isolated nonsingular solution of $f = 0$.  Rather than perform a singular value decomposition
on $Jf(x^*)$ as in Section~\ref{Sec:PseudoinverseRandomization},
we perform a (column pivoted) QR factorization of~$Jf(x^*)$,
that is, we compute
\begin{equation}\label{eq:QR}
Jf(x^*) = Q\cdot R\cdot P
\end{equation}
where $Q\in\bC^{n\times N}$ is unitary,
$R\in\bC^{N\times N}$ is upper triangular
and $P\in\bR^{N\times N}$ is a permutation matrix.
The permutation matrix swaps the columns which 
corresponds with simply reordering the variables.  
With this setup, we construct
the randomization matrix iteratively 
based on the largest 
values of the leverage scores of $Jf(x^*)$ as follows.

\begin{mydef}[Leverage score randomization matrix]\label{Def:LeverageScoreRandomization}
Following the setup as above, suppose
that \mbox{$\ell_1,\dots,\ell_n$}~are~the leverage scores of $Jf(x^*)$.  
Construct a reordering of the leverage scores 
so that they are in decreasing order, say 
$\ell_{k_1} \geq \ell_{k_2} \geq \cdots \geq \ell_{k_n} \geq 0$.
Fix $j_1 = 1$.  For $1 \leq r < N$, given $j_1 < \cdots < j_r$,
select $j_{r+1} > j_t$ to be the minimum value such 
that the $k_{j_1}, \dots, k_{j_{r+1}}$ rows of~$Jf(x^*)$
are linearly independent.  The corresponding 
{\em leverage score randomization matrix}
is $A\in\bR^{N\times n}$ which has the 
following~$N$ nonzero entries:
$$A_{r,k_{j_r}} = \|Jf(x^*)_{k_{j_r}}\|_2^{-1} \hbox{~~for~~} r = 1,\dots,N$$
where $Jf(x^*)_{p}$ is the $p^{\rm th}$ row
of $Jf(x^*)$.
\end{mydef}

\begin{example}
To illustrate, consider $x^* = (1,1)$ and 
$$f(x) = \left[\begin{array}{c} x_1 - 1 \\ x_1 - 1 \\ x_2 - 1 \\ x_2 - 1\end{array}\right] \hbox{~~so that~~} Jf(x^*) = \left[\begin{array}{cc} 1 & 0 \\ 1 & 0 \\ 0 & 1 \\ 0 & 1 \end{array}\right]
= \left[\begin{array}{cc} 1/\sqrt{2} & 0 \\
1/\sqrt{2} & 0 \\ 0 & 1/\sqrt{2} \\ 0 & 1/\sqrt{2} \end{array}\right]
\left[\begin{array}{cc} \sqrt{2} & 0 \\ 0 & \sqrt{2}\end{array}\right].
$$
The leverage scores of $Jf(x^*)$ are all equal,
namely $\ell_1 = \cdots = \ell_4 = 1/2$.
With the trivial ordering $k_j = j$, Definition~\ref{Def:LeverageScoreRandomization}
produces $j_1 = 1$ and $j_2 = 3$ since the first and
second rows of $Jf(x^*)$ are not linearly independent.  
Since each row of $Jf(x^*)$ has norm $1$, 
the leverage score randomized matrix~is 
$$A = \left[\begin{array}{cccc} 
1 & 0 & 0 & 0 \\
0 & 0 & 1 & 0\end{array}\right]
\hbox{~~~with~~~}
A\cdot f(x) = 
\left[\begin{array}{c} x_1 - 1 \\ x_2 - 1\end{array}\right].$$
\end{example}

\begin{proposition}
With the setup described above, if $A$ 
is a leverage score randomization matrix
for $Jf(x^*)$, then $A\cdot Jf(x^*)$
has rank $N$ such that each row has
unit length in the $2$-norm.
\end{proposition}
\begin{proof}
Since $Jf(x^*)$ has rank $N$, 
matrix $Q$ in \eqref{eq:QR}
has rank $N$.  In particular, at least
$N$ rows of both $Jf(x^*)$ and~$Q$ must be nonzero
with $\ell_{k_1} > 0$.
Moreover, since $R\cdot P$ is invertible, 
$Jf(x^*)_j = 0$ if and only if 
$Q_j = 0$, i.e., $\ell_j = 0$.  Thus, 
if $j_1,\dots,j_N$ are selected as in Defn.~\ref{Def:LeverageScoreRandomization},
then $\|Jf(x^*)_{k_{j_r}}\|_2 > 0$ for $r = 1,\dots,N$.
Hence, it follows that 
the rows of $A\cdot Jf(x^*)$ have
unit length in the the $2$-norm.

The fact that the resulting matrix
$A\cdot Jf(x^*)$ has rank $N$ using such a greedy selection 
of rows is classical in linear algebra and follows
from the dimension of the row span of $Jf(x^*)$ being $N$.
\end{proof}

\begin{example}\label{ex:TwistedCubic8}
With $h(x)$ as in \eqref{eq:TwistedCubicx0}
and $x^* = (1,-1,1,-1)$, the leverage scores of $Jh(x^*)$ are
$\ell_1 = \ell_2 = \ell_3 = 2/3$ and $\ell_4 = \ell_5 = 1$.
Regarding leverage scores as a measure of importance,
this shows that the fourth and fifth polynomials
in $h$, namely $x_0+x_1+x_2+x_3$ and $x_0-1$, respectively,
are equally the two most important at $x^*$.
The other three polynomials which define the twisted
cubic are equally important to each other, 
but less than the two linear polynomials.  
By taking the reordering
$k_1 = 4$, $k_2 = 5$, $k_3 = 1$, $k_4 = 2$, and $k_5 = 3$, we
have the randomized system
{\small
$$g(x) = \left[\begin{array}{ccccc}
0 & 0 & 0 & 1/2 & 0 \\
0 & 0 & 0 & 0 & 1 \\
1/\sqrt{6} & 0 & 0 & 0 & 0 \\
0 & 1/2 & 0 & 0 & 0 
\end{array}\right]\cdot h(x) = \left[\begin{array}{c}
(x_0 + x_1 + x_2 + x_3)/2 \\
x_0 - 1 \\
(x_0 x_2 - x_1^2)/\sqrt{6} \\
(x_1 x_2 - x_0 x_3)/2
\end{array}\right]$$
}so that each row of
{\small
$$Jg(x^*) = \left[\begin{array}{cccc}
1/2 & 1/2 & 1/2 & 1/2 \\
1 & 0 & 0 & 0\\
1/\sqrt{6} & 2/\sqrt{6} & 1/\sqrt{6} & 0\\
1/2 & 1/2 & -1/2 & -1/2 
\end{array}\right]$$
}has unit length in the $2$-norm
and its condition number with respect to the $2$-norm
is $8.8$.
\end{example}

Similar to the pseudoinverse randomization
when path tracking, one uses the leverage score randomization
based on the current point on the path 
and performs a predictor-corrector step with that randomized
system.  If the step is successful, the randomization 
is updated based on the new point on the path (see Section~\ref{Sec:Algorithm}).

The advantage of using a leverage score randomization 
is that the randomizing matrix 
is sparse and selects a well-constrained subset of the 
original polynomials.
Thus, one saves computational time by 
only evaluating the polynomials and their gradients
of the polynomials which are selected.

\section{Path tracking algorithms}\label{Sec:Algorithm}

We now aim to incorporate the patching and randomization
strategies into path tracking.  Suppose
that $F(x;p)$ is a parameterized system which
is polynomial in the variables $x\in X$ and analytic
in the parameters $p\in\bC^P$.  
Depending on the structure of $F$, we may regard
$X$ as a projective or affine space,
or, more generally, as a product of such spaces.  Suppose
that $p^*\in\bC^P$ is generic and~$S(p^*)$ 
consists of the isolated nonsingular solutions of $F(x;p^*) = 0$.
As mentioned in Section~\ref{Sec:ParameterHomotopies},
there is a generic behavior of $F$ 
with respect to the parameter space $\bC^P$
so that, for given~$\hat{p}\in\bC^P$, 
we can use the parameter 
homotopy $H$ defined in 
\eqref{eq:ParameterHomotopy} with start points $S(p^*)$
at $t = 1$ to compute $S(\hat{p})$, the 
set of all nonsingular isolated solutions of $F(x;\hat{p}) = 0$.
In particular, for each $x^*\in S(p^*)$, there is a smooth
homotopy path $x(t)$ for $t\in(0,1]$ such that
$x(1) = x^*$ and $H(x(t),t) \equiv 0$.
Since~$x(t)$ could be defined on products of 
projective and affine spaces, and~$F$ could be overdetermined, 
we utilize patching and randomization strategies
to track $x(t)$.  To avoid having to deal with
paths with divergent and singular endpoints, which
can be handled using endgames \cite[Chap.~10]{SW:Book},
we assume that $x(t)$ exists and is smooth on $[0,1]$
in Algorithm~\ref{Alg:PathTrack}.
By working intrinsically on the affine patch, one could 
attempt to reduce the linear algebra cost of performing a predictor-corrector~step.  

The justification for Algorithm~\ref{Alg:PathTrack} follows from
Bertini's Theorem (Theorem~\ref{thm:Bertini})
and the use of affine coordinate patches 
which permits computations regarding 
the path $x(t)$ to be performed using a well-constrained subsystem
on an affine space.  By having local control on the 
condition number, we aim to perform fewer operations
when path tracking
as exemplified~in~Section~\ref{Sec:Vision}.

\begin{algorithm}
\caption{Path Tracker}
\begin{algorithmic}
\STATE \textbf{Input:} A parameterized system $F(x;p)$ that is polynomial in $x\in X$ and analytic in $p\in\bC^P$ where $X$ is a product of projective and affine spaces; a generic parameter value $p^*\in\bC^P$ and a target parameter value $\hat{p}\in\bC^P$;
an isolated nonsingular singular $x^*$ of $F(x;p^*)=0$ such that the solution path $x(t)$ defined by $x(1) = x^*$ and $H(x(t),t)\equiv 0$ where $H$ as~in~\eqref{eq:ParameterHomotopy}
is smooth for~$t\in[0,1]$.
\medskip
\STATE \textbf{Output:} An isolated nonsingular solution in of $F(x;\hat{p})$ = 0.
\medskip
	\STATE{Initialize $z^* = x^*$, $q^* = p^*$, $t = 1$, and select an initial step size $dt > 0$, e.g., $dt = 0.1$.}
	\vspace{2mm}
	\WHILE{$t > 0$}
		\STATE{Apply a patching strategy (Section~\ref{Sec:Patch}), possibly adaptively using the current point $z^*$, 
to the projective spaces in $X$ so that all computations are performed on an affine space.  Update $z^*$ to lie on the selected patches.  Construct $G(x;p)$ which is $F(x;p)$ together with the added patch equations.}
\medskip
		\STATE{Apply a randomization strategy (Section~\ref{Sec:Randomization}), possibly adaptively using the current point~$z^*$ and parameter value $q^*$, to $G(x;p)$ to create a well-constrained subsystem~$A\cdot G(x;p)$ such that $A\cdot J_xG(z^*;q^*)$ is nonsingular, where $J_xG(x;p)$ is the Jacobian matrix of $G$ with respect~to~$x$.}
		\vspace{2mm}
		\STATE{Construct the homotopy $H(x,t) = A\cdot G(x;tp^* + (1-t)\hat{p})$ and perform a predictor-corrector step from $t$ to $\max\{t-dt,0\}$ using the homotopy $H$ with start point $z^*$ at $t$ yielding $y^*$.}
		\vspace{2mm}
		\IF{predictor-corrector step is successful}
			\STATE{Update $z^* = y^*$, $t = \max\{t - dt,0\}$, and $q^* = tp^* + (1-t)\hat{p}$.}
			\vspace{2mm}
			\STATE{Consider increasing the step size if multiple successful steps in a row, e.g., update $dt = 2 \cdot dt$ if $3$ consecutive successful steps.}
			\vspace{2mm}
			\STATE{Consider applying early truncation (see Section~\ref{Sec:Truncation}).}
		\ELSE
			\STATE{Decrease the step size, e.g., set $dt = dt/2$.}
		\ENDIF
	\ENDWHILE
\end{algorithmic}
\label{Alg:PathTrack}
\end{algorithm}

\begin{example}\label{ex:TwistedCubic5}
To illustrate Algorithm~\ref{Alg:PathTrack}, 
we consider the parameterized system
$$F(x;p) = \left[\begin{array}{c}
x_0 x_2 - x_1^2 \\
x_1 x_2 - x_0 x_3 \\
x_1 x_3 - x_2^2 \\
x_2 + p_1x_0 + p_2 x_1 + p_3x_3
\end{array}\right]$$
where $x\in X = \bP^3$ and $p\in\bC^3$.  Thus,
$F(x;p)=0$ defines the intersection of the twisted cubic
with a parameterized family of hyperplanes
which clearly has $3$ nonsingular isolated solutions
generically.  In the following, we consider tracking
the 3 solution paths as $p^* = (1,1,1)$
deforms to \mbox{$\hat{p} = (-1,0.1i,0)$}, where $i = \sqrt{-1}$,
starting with the 3 points in \eqref{eq:ThreePoints}.
This setup ensures that we will need to use 
an affine patch followed by a randomization.

For the fixed
general affine patch, we used the 
randomly selected patch
\begin{equation}\label{eq:FixedPatchExample}
(0.3509 + 0.1476i)x_0 + (0.4524 - 0.4487i)x_1
- (0.4159 + 0.2470i)x_2 + (0.4609 + 0.0523i)x_3 = 1
\end{equation}
where $i = \sqrt{-1}$.

In Figures~\ref{fig:TwistedCubic1},~\ref{fig:TwistedCubic2}, and~\ref{fig:TwistedCubic3}, we compare the condition
number with respect to the $2$-norm 
of the Jacobian matrix along the three paths using
the different patching and randomization strategies. 
For the fixed randomization, we used $A = [I~Q]\in\bC^{4\times5}$
with the randomly selected 
$$Q = [0.1792-0.1432i, -0.7159-0.5784i, 0.1866-0.4692i, 0.4524+0.9864i]^T.$$ 

\begin{figure}[!ht]
\centering
\includegraphics[scale=0.4]{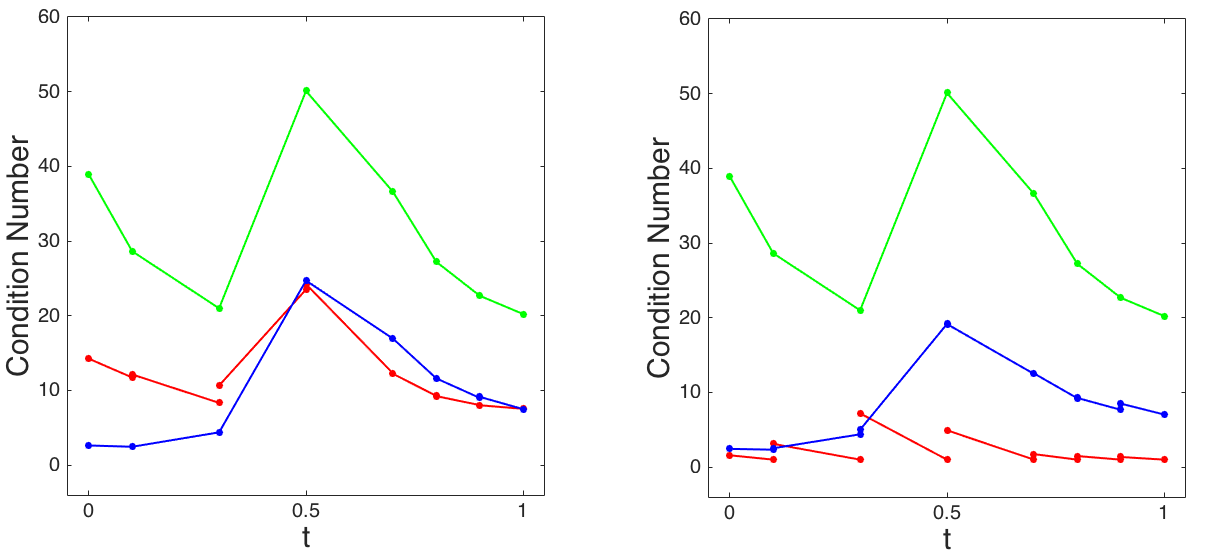}

~~~~(a) \hspace{2.37in} (b)

\caption{Plot of the condition numbers
along the first path.  (a) Using a fixed randomization,
Green: fixed affine patch, Red: orthogonal affine patch,
Blue: coordinate-wise affine patch. 
(b) Green: fixed randomization with fixed affine patch, 
Red: pseudoinverse randomization with orthogonal affine patch,
Blue: leverage score randomization with coordinate-wise affine patch.}
\label{fig:TwistedCubic1}
\end{figure}

\begin{figure}[!ht]
\centering
\includegraphics[scale=0.47]{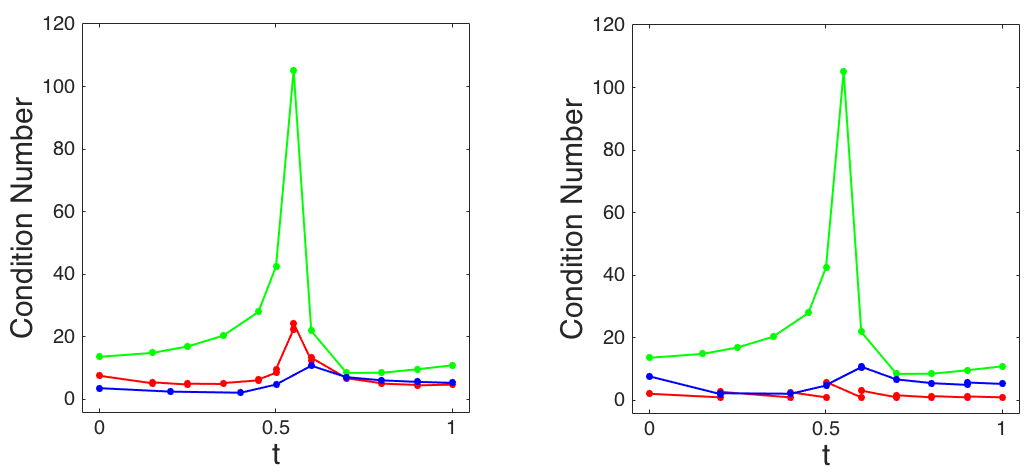}

~~~~~~~(a) \hspace{2.4in} (b)

\caption{Plot of the condition numbers
along the second path.  (a) Using a fixed randomization,
Green: fixed affine patch, Red: orthogonal affine patch,
Blue: coordinate-wise affine patch. 
(b) Green: fixed randomization with fixed affine patch, 
Red: pseudoinverse randomization with orthogonal affine patch,
Blue: leverage score randomization with coordinate-wise affine patch.}
\label{fig:TwistedCubic2}
\end{figure}

\begin{figure}[!ht]
\centering
\includegraphics[scale=0.4]{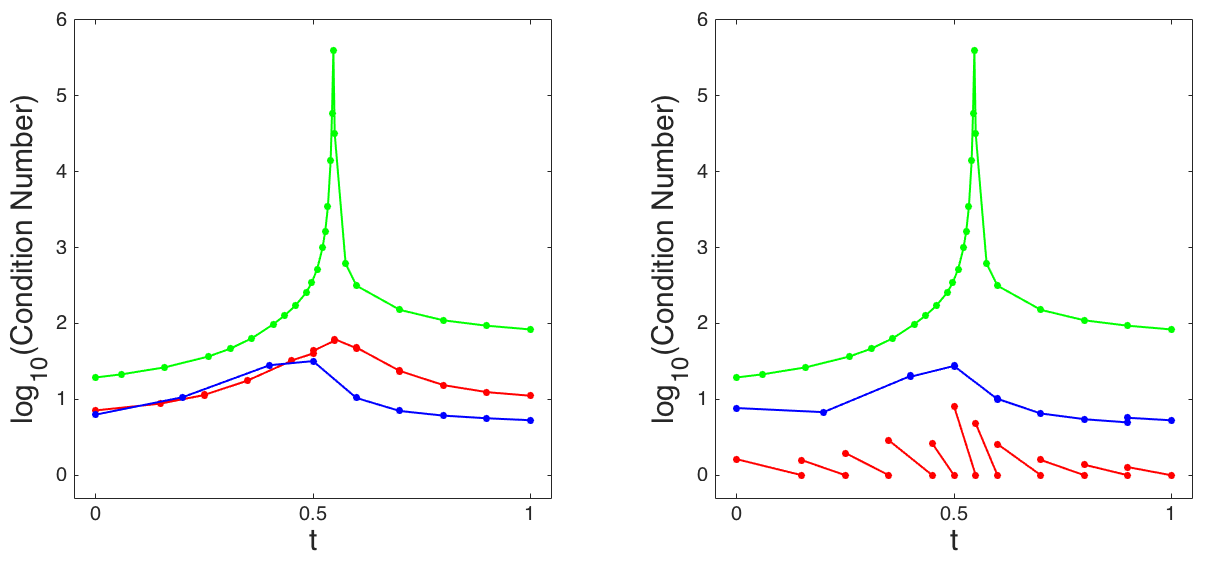}

~~~~~(a) \hspace{2.37in} (b)

\caption{Plot of the logarithm of the condition numbers
along the third path.  (a) Using a fixed randomization,
Green: fixed affine patch, Red: orthogonal affine patch,
Blue: coordinate-wise affine patch. 
(b) Green: fixed randomization with fixed affine patch, 
Red: pseudoinverse randomization with orthogonal affine patch,
Blue: leverage score randomization with coordinate-wise affine patch.}
\label{fig:TwistedCubic3}
\end{figure}
\end{example}

\section{Truncation}\label{Sec:Truncation}

In many applications, one is interested
in computing real solutions.  For a parameterized
system, we are typically deforming from a complex 
parameter value $p^*$ to a real parameter value~$\hat{p}$
and we would like a heuristic approach that could help identify
which paths are headed to nonreal endpoints.
The paths with (potentially) nonreal endpoints will be truncated
to limit wasted~computation.  

The idea of our proposed test is to consider 
how the size of the imaginary part of the points on the path
is changing with respect to $t$.  Since real endpoints
have imaginary part equal to $0$, we want to make a heuristic
decision based on the data along the path to decide
if the imaginary part could reasonably limit to $0$.
If this is not reasonable, then we consider the path to be heading towards a nonreal endpoint.  There is a trade-off between
when to start applying this test.  If the test is applied far from $t = 0$, then the imaginary part will be significantly 
impacted by the starting parameter $p^*$.  If the test is applied
very close to $t = 0$, then there is little computational savings in truncation.  In our experiments, we start testing
when $t < 0.3$.  
See Algorithm~\ref{Alg:PathTrack} for the location of truncation
in the path tracking algorithm.

Our proposed truncation test takes as input 
two points along the path, say $x(t_1)$ and $x(t_2)$ where $0 < t_2 < t_1$
and makes a decision based on the angle between the following two lines:
\begin{itemize}
\item line connecting $(t_1,\|\imag x(t_1)\|_2)$ and $(t_2,\|\imag x(t_2)\|_2)$, and
\item line connecting $(t_2,\|\imag x(t_2)\|_2)$ and $(0,0)$
\end{itemize}
as shown in Figure~\ref{fig:EarlyTrunc}.
A large angle suggests that it is realistic to believe
that the path is heading towards a nonreal endpoint.
In our testing, we considered ``large'' to be at least
$\frac{5\pi}{6}$ in which case that path was truncated
from further computation.

\begin{figure}[!ht]
\centering
\begin{subfigure}{.38\textwidth}
\includegraphics[scale = 0.368]{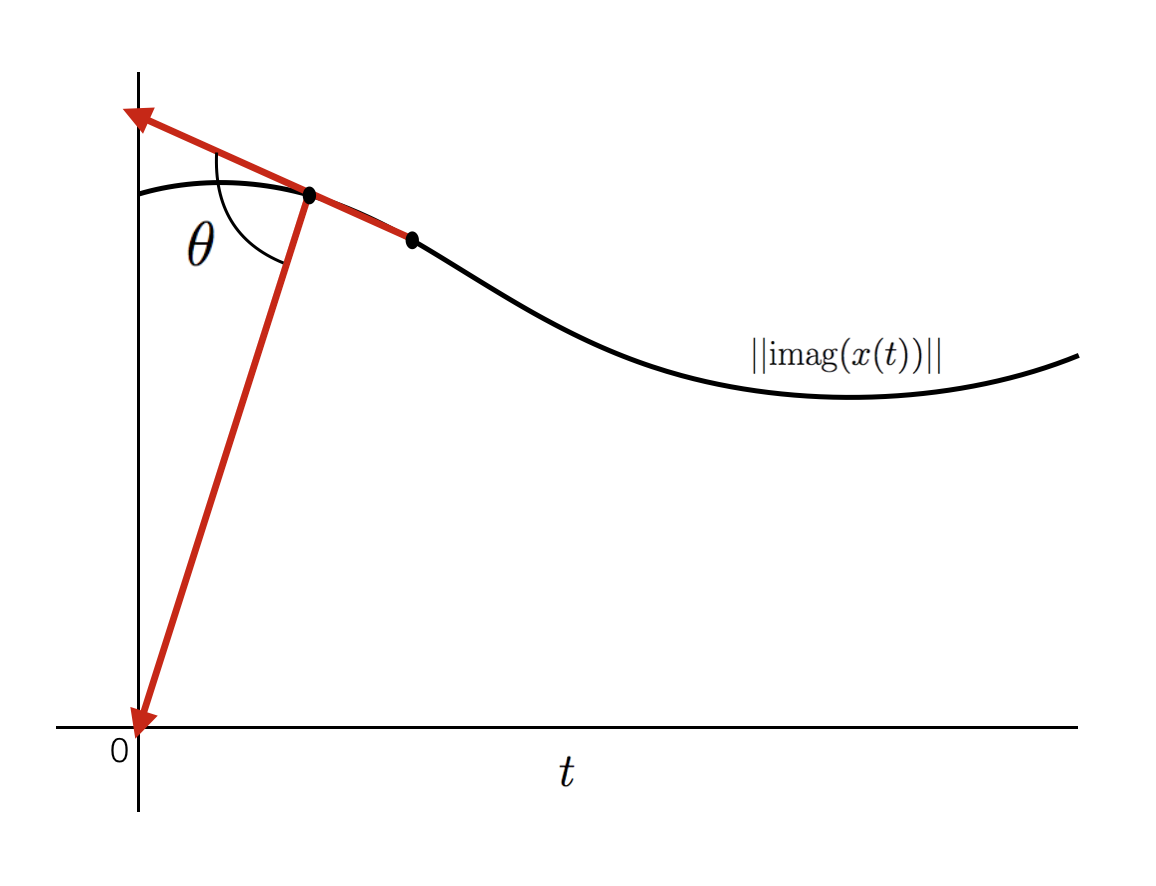}
\end{subfigure}
\hspace{10mm}
\begin{subfigure}{.38\textwidth}
  \includegraphics[scale = 0.36]{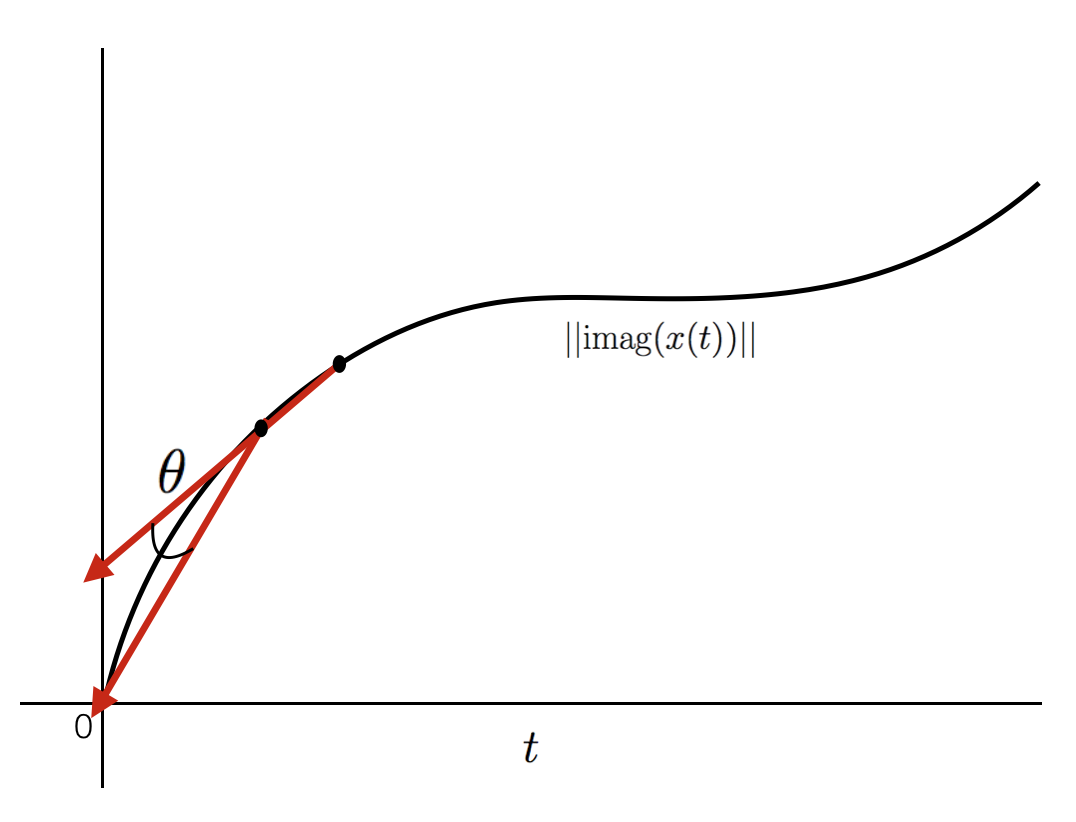}
\end{subfigure}

~~~~~~~(a) \hspace{2.5in} (b)

\caption{Schematic representation of (a) large angle leading toward a non-real solution and (b) small angle that leads to a real solution.}
\label{fig:EarlyTrunc}
\end{figure}

\section{Applications in computer vision}\label{Sec:Vision}

In order to demonstrate the three aspects of homotopy
continuation investigated, namely adaptive affine coordinate patches, adaptively selected well-constrained subsystems,
and truncation for paths heading to nonreal solutions, 
we consider two problems in computer vision.  
These problems are so-called {\em minimal problems}
in computer vision in that they generically have finitely 
many solutions with the current techniques
unstable for small problem sizes as well as require many assumptions and simplifications \cite{Kukelova,KukelovaPajdla}.  
We describe two minimal problems in Sections~\ref{Sec:5pt} 
and~\ref{Sec:6pt}, with Section~\ref{Sec:Results}
presenting the computational results.
 
\subsection{5-point problem}\label{Sec:5pt}

\begin{figure}
	\centering
	\includegraphics[height=2in, width=3.5in]{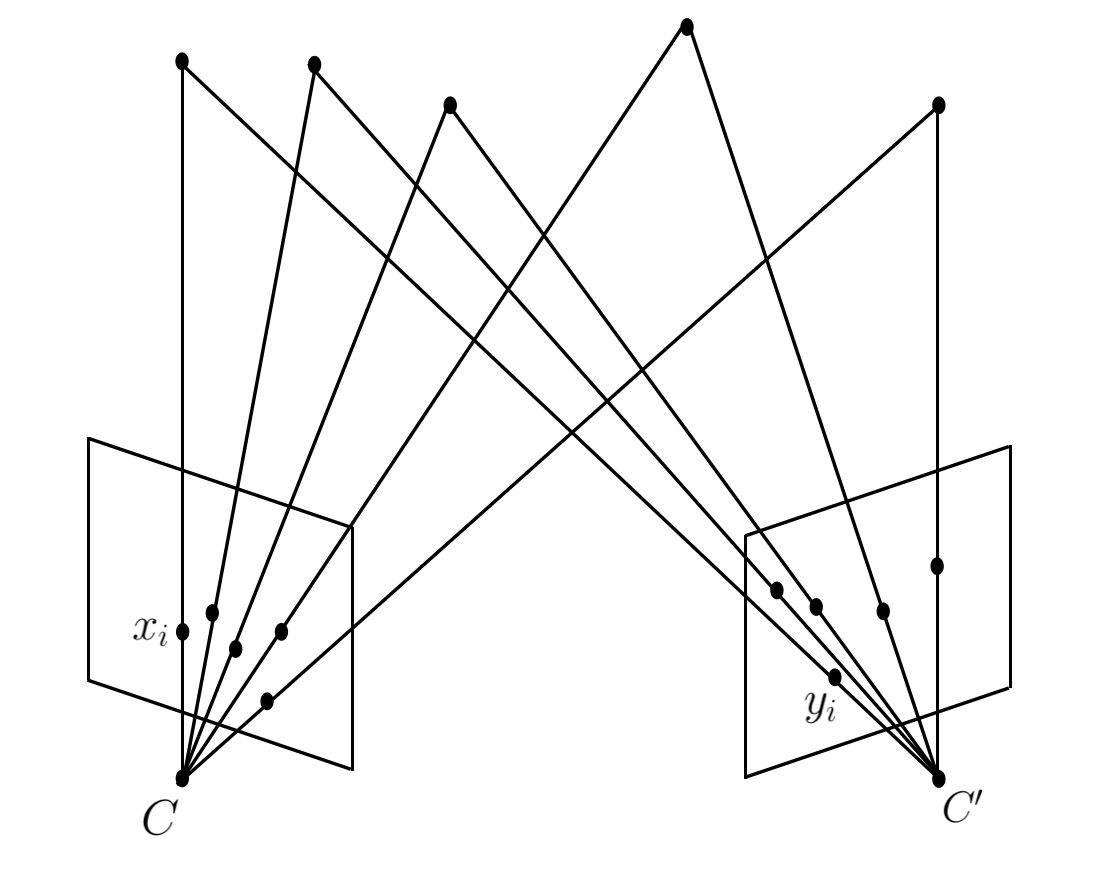}
	\caption{Schematic drawing of the 5-point image reconstruction problem.}\label{fig:FivePoint}
\end{figure}

The 5-point problem involves two cameras where
both images have 5 corresponding points mapping 
from a 3D object in space \cite{Kruppa,Faugeras,Nister,Stewenius1}.   
With the 5 corresponding point pairs $x_i$ for camera $C$ and 
$y_i$ for camera $C'$ as shown in Figure~\ref{fig:FivePoint}, this problem computes the relative position and orientation of two calibrated cameras using
the polynomial system:
\[\left[\begin{array}{cc}
            2EE^TE - \textrm{trace}(EE^T)E & \\
            y_i^TEx_i &  i = 1,\dots,5
\end{array}\right].\]
As written, this system consists of 14 polynomials defined on $\bP^8$ so that it is an overdetermined parameterized system
defined on projective space which generically has 10 solutions.

\subsection{6-point problem}\label{Sec:6pt}

\begin{figure}
	\centering
	\includegraphics[scale = 0.5]{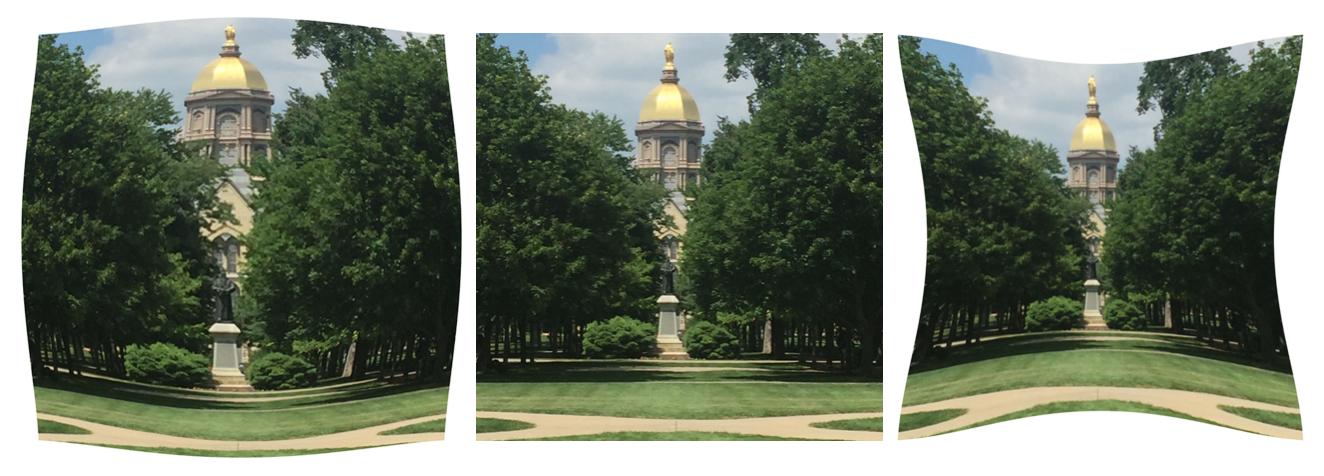}
	\caption{Image reconstruction with radial distortion.}
	\label{fig:SixPoint}
\end{figure}

The 6-point problem involves two cameras where
both images have 6 corresponding points mapping 
from a 3D object in space with a radial distortion parameter 
$\lambda$ \cite{Stewenius2,KukelovaPajdla,Byrod}.
Figure~\ref{fig:SixPoint} shows images which have radial distortion.
With the 6 corresponding point pairs $x_i$ for camera $C$ and 
$y_i$ for camera $C'$, 
this problem computes the relative position and orientation of two 
calibrated cameras using the polynomial system:
\[\left[\begin{array}{cc}
            2EE^TE - \textrm{trace}(EE^T)E & \\
            p_i(\lambda)^TEq_i(\lambda) & i = 1,\dots,6
\end{array}\right]\]
where
\begin{equation}
p_i(\lambda) = \left[ \begin{array}{cc}
y_i \\
1 + \lambda||y_i||_2^2
\end{array}\right] \hspace{8mm}\hbox{and}\hspace{8mm}
q_i(\lambda) = \left[ \begin{array}{cc}
x_i \\
1 + \lambda||x_i||_2^2
\end{array}\right].
\end{equation}
As written, this system consists of 15 polynomials defined
on $\bP^8\times\bC$ so that it is an overdetermined
parameterized system defined on a product
of a projective space and an affine space which
generically has 52 solutions.

\subsection{Results}\label{Sec:Results}

We implemented Algorithm~\ref{Alg:PathTrack} 
to solve the 5-point and 6-point problems\footnote{Available at \url{http://dx.doi.org/10.7274/R0C53HXK}.} using
homotopy continuation with the data resulting from
solving 100 random instances 
are summarized Tables~\ref{tab:5pt} and~\ref{tab:6pt}, respectively.  In Algorithm~\ref{Alg:PathTrack}, the predictor-corrector step we utilized was 
the classical $4^{\rm th}$ order Runge-Kutta predictor
with a maximum of 3 Newton iterations.  
These tables summarize average number of steps per path,
average number of arithmetic operations per parameter instance
to be solved, and average time (in seconds) per instance
for various patching (Section~\ref{Sec:Patch})
and randomization (Section~\ref{Sec:Randomization}) 
strategies together with path truncation (Section~\ref{Sec:Truncation}).  

\begin{table}
\centering
\begin{tabular}{l | c | c | c | c | c | c}
 & FP/FR & OP/FR & CWP/FR & OP/PIR & CWP/LSR & CWP/LSR/ET  \\
\hline \hline
Avg. \# steps/path & 14.127 & 11.261 & 10.026 & 10.859 & 9.224 & 8.338 \\ 
Avg. \# operations & 53,646 & 53,886 & 34,733 & 41,038 & 30,827 & 44,908\\
Avg. time (sec) &  0.3641 & 0.4449 & 0.3611 & 0.4637 & 0.2862 & 0.2608
\end{tabular}
\caption{Summary of solving $100$ random instances of the 5-point problem}
\label{tab:5pt}
\end{table}

\begin{table}
\centering
{\small
\begin{tabular}{l | c | c | c | c | c | c}
 & FP/FR & OP/FR & CWP/FR  & OP/PIR & CWP/LSR & CWP/LSR/ET \\
\hline \hline
Avg. \# steps/path & 29.667 & 23.361 & 24.717 & 20.224 & 21.721 & 15.746 \\ 
Avg. \# operations & 11,649,000 & 9,029,600 & 7,289,400 & 7,740,500 & 6,340,300 & 4,700,900\\
Avg. time (sec) &  5.810 & 4.526 & 5.302 & 4.259 & 4.832 & 3.707
\end{tabular}
}
\centering
\caption{Summary of solving $100$ random instances of the 6-point problem}
\caption*{FP = fixed random patch (Section~\ref{Sec:GlobalRandomPatch}), 
OP = adaptive orthogonal patch (Section~\ref{Sec:Orthogonal}), 
CWP~=~adaptive coordinate-wise patch (Section~\ref{Sec:LocalCoordPatch}), 
FR = fixed randomization (Section~\ref{Sec:FixedRandomization}), 
PIR~=~adaptive pseudoinverse randomization (Section~\ref{Sec:PseudoinverseRandomization}), 
LSR = adaptive leverage score randomization (Section~\ref{Sec:LeverageRandomization}), and 
ET = early truncation (Section~\ref{Sec:Truncation}).}
\label{tab:6pt}
\end{table}

The results show that the combination of 
adaptive coordinate-wise patching, 
leverage score randomization, and early truncation 
had both the lowest average number of steps per path
and average solving time for both the 5-point and 6-point problems.  
In particular, for the 5-point problem, 
this combination yielded about a $40\%$ decrease in average number of steps per path over tracking using 
a fixed random patch with fixed randomization.  For the 6-point problem, the decrease
was $47\%$ due to the larger size and more paths
which could be truncated.  
In particular, every adaptive method results in
reducing the average number of steps per path
compared with using a fixed random patch with 
fixed randomization.

\section{Conclusion}\label{Sec:Conclusion}

Three aspects of using homotopy continuation to solve parameterized 
systems were investigated: selecting affine coordinate patches, 
selecting well-constrained subsystems, and truncating paths
which appear to be ending at nonreal solutions.  
The results in Section~\ref{Sec:Vision} demonstrate
substantial improvement by using adaptive selections
over fixed random choices in path tracking.  

\section{Acknowledgments}\label{Sec:Acknowledgments}

The authors would like to thank Sameer Agarwal for his expertise in minimal problems, their implementation, and suggestions regarding the use of numerical algebraic geometry.  We would also like to thank Alan Liddell, Ilse Ipsen, and Tim Kelley for helpful discussions.

\newcommand{\noopsort}[1]{}\def\cprime{$'$}


\begin{thebibliography}{10}

\bibitem{AlKhateeb} A.N. Al-Khateeb, J.M. Powers, S. Paolucci, A.J. Sommese, J.A. Diller, J.D. Hauenstein, and J. Mengers,
One-dimensional slow invariant manifolds for spatially homogeneous reactive systems.
\emph{J. Chem. Phys.}, 131, 2009.

\bibitem{Paramotopy} D.J. Bates, D.A. Brake, and M.E. Niemerg, Paramotopy: Parameter homotopies in parallel.
Available at {\tt paramotopy.com}.

\bibitem{BHSW:Bertini} D.J. Bates, J.D. Hauenstein, A.J. Sommese, and C.W. Wampler,
\newblock Bertini: Software for numerical algebraic geometry.
\newblock Available at {\tt bertini.nd.edu}.

\bibitem{BHSW:BertiniBook} D.J. Bates, J.D. Hauenstein, A.J. Sommese, and C.W. Wampler,
Numerically Solving Polynomial Systems with Bertini. 
\emph{SIAM}, 2013.

\bibitem{Byrod} M. Byr\"{o}d, K. Josephson, and K. \AA str\"{o}m, Improving numerical accuracy of Gr\"{o}bner basis
polynomial equation solvers. In \emph{ICCV07}, IEEE, 
2007, pp. 1--8.

\bibitem{Demazure} M. Demazure,
Sur deux problemes de reconstruction. 
\emph{INRIA}, 1988.

\bibitem{Deuflhard} P. Deuflhard,
Newton Methods for Nonlinear Problems: Affine Invariance and Adaptive Algorithms.
\emph{Springer Series in Computational Mathematics}, 35, 2004.

\bibitem{Faugeras} O.D. Faugeras and S.J. Maybank, Motion from point matches: multiplicity of solutions.
\emph{International Journal of Computer Vision}, 4(3):225--246, 1990.

\bibitem{RANSAC} M.A. Fischler and R.C. Bolles,
Random sample consensus: a paradigm for model fitting 
with applications to image analysis and automated cartography. \emph{Commun. ACM},
24(6):381--395, 1981.

\bibitem{alphaCertified} J.D. Hauenstein and F. Sottile,
Algorithm 921: alphaCertified: Certifying solutions to polynomial systems.  {\em ACM Trans. Math. Softw.}, 38(4):28, 2012. 

\bibitem{IsosingularDeflation} J.D. Hauenstein and C.W. Wampler,
Isosingular sets and deflation.
\emph{Found. Comput. Math.}, 13(3):371--403, 2013.

\bibitem{3RSynthesis} J.D. Hauenstein, C.W. Wampler, and M. Pfurner,
Synthesis of three-revolute spatial chains for body guidance.
\emph{Mechanism and Machine Theory}, 110:61--72, 2017.

\bibitem{Leverage}
D.C. Hoaglin and R.E. Welsch,
The hat matrix in regression and ANOVA. 
\emph{The American Statistician}, 32(1):17--22, 1978.

\bibitem{Ipsen} I.C.F. Ipsen and T. Wentworth,
The effect of coherence on sampling from matrices with orthonormal columns, and preconditioned least squares problems. \emph{SIAM J. Matrix Anal. Appl.}, 35(4):1490--1520, 2014. 

\bibitem{Irschara} A. Irschara, C. Zach, J. M. Frahm, and H. Bischof, 
From structure-from-motion point clouds to fast location recognition.
\emph{CVPR}, 2599--2606, 2009.

\bibitem{Kruppa} E. Kruppa, Zur Ermittlung eines Objektes aus zwei Perspektiven mit innerer Orientierung.
Sitzungsberichte der Mathematisch Naturwissenschaftlichen Kaiserlichen Akademie der
Wissenschaften, 122:1939--1948, 1913.

\bibitem{Kukelova} Z. K\'{u}kelov\'{a}, 
Algebraic Methods in Computer Vision.  
\emph{Czech Technical University}, 2013.

\bibitem{KukelovaPajdla} Z. K\'{u}kelov\'{a} and T. Pajdla,
Two minimal problems for cameras with radial distortion.  In \emph{ICCV07}, IEEE, 2007, pp. 1--8.

\bibitem{Leibe} B. Leibe, N. Cornelis, K. Cornelis, and L. Van Gool,
Dynamic 3D scene analysis from a moving vehicle. 
In \emph{CVPR}, IEEE, 2007, pp. 1--8.

\bibitem{Deflation1} A. Leykin, J. Verschelde, and A. Zhao,
Newton's method with deflation for isolated singularities of polynomial systems. 
\emph{Theoret. Comput. Sci.} 359(1-3):111--122, 2006.

\bibitem{Morgan} A.P. Morgan,
A transformation to avoid solutions at infinity for polynomial systems. 
\emph{Appl. Math. Comput.}, 18(1):77--86, 1986. 

\bibitem{ChemApp1} A.P. Morgan, Solving Polynomial Systems Using Continuation for Engineering and Scientific Problems.
\emph{Prentice-Hall}, 1987.
    
\bibitem{CoeffParam} A.P. Morgan and A.J. Sommese,
Coefficient-parameter polynomial continuation.
{\em Appl. Math. Comput.}, 29(2):12--160, 1989.

\bibitem{Nister} D. Nist\'{e}r, An efficient solution to the five-point relative pose problem. \emph{IEEE Transactions
on Pattern Analysis and Machine Intelligence}, 26(6):756--777, 2004.
        
\bibitem{ShubSmale} M. Shub and S. Smale, 
Complexity of B\'ezout's theorem. I: Geometric aspects.
\emph{J. Amer. Math. Soc.}, 6(2):459--501, 1993. 

\bibitem{Snavely} N. Snavely, S. M. Seitz, and R. Szeliski,
Photo tourism: exploring photo collections in 3D.
\emph{ACM SIGGRAPH}, 835--846, 2006.

\bibitem{NAG} A.J. Sommese and C.W. Wampler,
Numerical algebraic geometry.
{\em Lectures in Appl. Math.}, 32:749--763, 1996.

\bibitem{SW:Book} A.J. Sommese and C.W. Wampler,
The Numerical Solution of Systems of Polynomials Arising in Engineering and Science.
\emph{World Scientific}, 2005. 

\bibitem{Stewenius1} H. Stew\'{e}nius, C. Engels, and D. Nist\'{e}r, Recent developments on direct relative orientation.
\emph{ISPRS Journal of Photogrammetry and Remote Sensing}, 60(4):284--294, 2006.

\bibitem{Stewenius2} H. Stew\'{e}nius, D. Nist\'{e}r, F. Kahl, and F. Schaffalitzky, A minimal solution for relative pose
with unknown focal length. \emph{Image and Vision Computing}, 26(7):871--877, 2008.

\bibitem{Wampler} C.W. Wampler and A.J. Sommese, 
Numerical algebraic geometry and algebraic kinematics.
\emph{Acta Numerica}, 20:469--567, 2011.

\end{thebibliography}
\end{document}